\documentclass{amsart}

\usepackage[latin1]{inputenc}
\usepackage{latexsym}
\usepackage{amsfonts}
\usepackage{amssymb}
\usepackage{color}
\usepackage{varioref}
\usepackage{epsfig}
\usepackage{xy}

\newtheorem{thm}{Theorem}[section]

\newtheorem{prop}[thm]{Proposition}
\newtheorem{lemma}[thm]{Lemma}
\newtheorem{cor}[thm]{Corollary}

          {\theoremstyle{definition}
}
          {\theoremstyle{definition}
}

\DeclareMathOperator{\Hess}{Hess}

\newcommand{\RR}{{\mathbb R}}

\newcommand{\PP}{{\mathcal P}}

\begin{document}

\title{On the number of connected components of the parabolic curve}
\date{\today}
\author{Benoît Bertrand}
\address{Institut mathématique de Toulouse, I.U.T. de Tarbes, 1 rue Lautréamont, BP 1624, 65016 Tarbes}
\email{benoit.bertrand@iut-tarbes.fr}
\author{Erwan Brugall\'e}
\address{Univ.\ Paris 6, 175 rue du Chevaleret, 75 013 Paris, France}
\email{brugalle@math.jussieu.fr}

\subjclass[2000]{}

\begin{abstract}
 We construct a polynomial of degree $d$ in two variables whose
 Hessian curve has $(d-2)^2$ connected components using Viro
 patchworking. In particular, this implies the existence of a smooth
 real algebraic surface of degree $d$ in $\RR P^3$ whose parabolic
 curve is smooth and has $d(d-2)^2$ connected components.  
\end{abstract}

\maketitle

\hspace{50ex}
\textit{``Salut à toi Che Guevara''}

\hspace{63ex}
\textit{Bérurier Noir}
\vspace{3ex}

\section{Introduction}
The Hessian of a polynomial $P(X_0,\ldots,X_n)$ is the determinant of
the matrix $(\frac{\partial^2 P}{\partial X_i\partial X_j})_{0\le
  i,j\le n}$. If $P$ is of degree $d$, then this determinant, denoted
by $\Hess(P)$, is generically a polynomial of degree $(n+1)(d-2)$. In
this note, we are interested in the real solutions of the system $\PP
(P) = \{P=0\}\cap\{ \Hess(P)=0\}$ when $P$ is a generic polynomial
with real coefficients.

In the case when $n=2$ and $P(X,Y,Z)$ is homogeneous, it is well known
that the set $\PP (P)$ is the set of real flexes of the curve with
equation $P(X,Y,Z)=0$. In \cite{Kle2} (see also \cite{Ron1},
\cite{Schuh}, and \cite{V10}), Klein proved that the number of real
flexes of a smooth real algebraic curve of degree $d$ cannot exceed
$d(d-2)$, and that this bound is sharp.

In the case when $n=3$ and $P(X,Y,Z,T)$ is homogeneous, not much is
known about the curve $\PP (P)$, called the \textit{parabolic curve}
of the surface with equation $P=0$.

  If $P$ is of degree $d$, then according to Harnack
inequality the curve $\PP (P)$ cannot have more than
$2d(d-2)(5d-12)+2$ connected components. 
Arnold's problem \cite[2001-2]{Arn2} on the topology of the parabolic
curve asks in particular for the maximal number of components of $\PP
(P)$ or at least for it asymptotic (See also problem 2000-2).
Ortiz-Rodriguez constructed in \cite{Rod1} smooth real algebraic
surfaces of any degree $d\ge 3$ whose parabolic curve is smooth and
has $\frac{d(d-1)(d-2)}{2}$ connected components.

Her construction uses auxiliary parabolic curves of graphs of
polynomials (i.e. $P(X,Y,Z,1)=Z-Q(X,Y)$). In this case, the curve $\PP
(P)$ is the locus of the graph where the Gaussian curvature vanishes
and its projection to the plane $(X,Y)$ has equation $\Hess(Q)=0$
(note that $Q$ and $\Hess(Q)$ are not necessarily homogeneous). If $Q$
is of degree $d$, then $\Hess(Q)$ is (generically) of degree $2(d-2)$
and defines a curve with at most $(2d-5)(d-3)+1$ compact connected
components in $\RR^2$. The maximal number of compact connected
components of such a curve in this special case is the subject of
problem \cite[2001-1]{Arn2} (See also problem 2000-1).  In
\cite{Rod1}, Ortiz-Rodriguez constructed real polynomials $Q(X,Y)$ of
degree $d\ge 3$ whose Hessian define smooth real curves with
$\frac{(d-1)(d-2)}{2}$ compact connected components in $\RR^2$. In
small degrees, this construction has been slightly improved in
\cite{Sot1}.

Note that if $Q(X,Y)$ is a polynomial in two variables, then 
$$\Hess(Q)=\frac{\partial^2 Q}{\partial X^2}\frac{\partial^2
  Q}{\partial Y^2} - \left(\frac{\partial^2 Q}{\partial X\partial
  Y}\right)^2$$

In this note, we prove the following result
which improves the previously known asymptotic by a factor 2.

\begin{thm}\label{theorem}
For any $d\ge 4$, there exists a real polynomial $Q_d(X,Y)$ of degree
$d$ such that the curve with equation $\Hess(Q_d)=0$ is smooth and has $(d-2)^2$
compact connected components in $\RR^2$.
\end{thm}
Theorem \ref{theorem} is proved in section \ref{proof}. The main tool
is Viro's Patchworking theorem to glue Hessian curves (see section
\ref{glue}).

\begin{cor}
For any $d\ge 4$, there exists a smooth real algebraic surface in $\RR
P^3$ of degree $d$ whose parabolic curve is smooth and has $d(d-2)^2$
connected components. 
\end{cor}

\begin{proof}
As observed in \cite[Theorem 5]{Rod1}, if the real curve $\Hess(Q)(X,Y)=0$ has
$k$ compact connected 
components in $\RR^2$, then for $\varepsilon$ small enough, the
parabolic curve of 
the surface with equation $R(Z)-\varepsilon Q(X,Y)=0 $, where $R(Z)$
is a real polynomial of degree $d$ with $d$ distinct real roots, has
$dk$ connected components.
\end{proof}

\section{Gluing of Hessians}\label{glue}

Let $Q_t(X,Y)=\sum a_{i,j}(t)X^iY^j$ be a polynomial   whose
coefficients are real polynomials in  
one variable $t$. Such a polynomial has two natural Newton polytopes
depending on whether $Q_t(X,Y)$ is considered as a polynomial in the
variables $X$ and $Y$ or as a polynomial in
 $X$, $Y$ and $t$. We denote by $\Delta_2(Q_t)$
the former Newton polytope, and by $\Delta_3(Q_t)$ the latter.
There
exists a convex piecewise linear function $\nu:\Delta_2(Q_t)\to\RR$
whose graph is the union of the bottom faces of $\Delta_3(Q_t)$. The
linearity domains of the function $\nu$ induce a subdivision
$\tau_\nu$ of $\Delta_2(Q_t)$. Note that since real numbers are constant
real polynomials, this construction makes sense also for
polynomials in $\RR [X,Y]$, in this case the function $\nu$ is constant
and the subdivision $\tau$ is the trivial one.

Let $\Delta'$ be a cell of $\tau_\nu$. The restriction of $\nu$ to
 $\Delta'$ is given by a linear function $L:(i,j)\mapsto \alpha i
 +\beta j +\gamma$ which does not coincide with $\nu$ on any polygon
 of $\tau_\nu$ strictly containing $\Delta'$.  If $\Delta'$ is of
 dimension $k \le 2$ there is a $2-k$ dimensional family of such
 functions but the following construction does not depend on the
 choice of the function as long as $\Delta_3(Q_t) \setminus \nu
 (\Delta')$ is strictly above the graph of $L$.  We define the
 $\Delta'$-truncation of $Q_t(X,Y)$ as the polynomial
 $Q^{\Delta'}(X,Y)$ in $\RR[X,Y]$ given by substituting $t=0$ in the
 polynomial $t^{-\gamma}Q_t(t^{-\alpha}X,t^{-\beta}Y)$.

Viro's Patchworking Theorem asserts that if all the polynomials
$Q^{\Delta'}(X,Y)$ are non singular in $(\RR^*)^2$ when $\Delta'$ goes
through all the cells of $\tau_\nu$, then for a small enough real
number $t$, the real algebraic curve with equation $Q_t(X,Y)=0$ is a
gluing of the real algebraic curves with equation $Q^{\Delta'}(X,Y)=0$
when $\Delta'$ goes through all the $2$-dimensional polygons of
$\tau_\nu$.  In particular any compact oval in $(\RR^*)^2$ of a curve
defined by $Q^{\Delta'}(X,Y)$ leads to an oval of the curve defined by
$Q_t(X,Y)$ when $t$ is small enough.  For a more precise statement of
Viro's Patchworking Theorem, we refer to \cite{V1}, \cite{V4}, or
\cite{Ris}. See also \cite{Mik12} for a tropical approach.

The key observation of this paper is that when gluing the polynomials
$Q^{\Delta'}(X,Y)$ one also glues their Hessians.  We formalize this
in the following proposition.

Consider as above a polynomial $Q_t(X,Y)$ whose coefficients are real
polynomials.  Let us denote by $\widetilde \nu$ the convex piecewise
linear function constructed as above out of the Hessian $\Hess(Q_t)$
of $Q_t(X,Y)$ with respect to the variables $X$ and $Y$.  If $\Delta'$
is a polygon of $\tau$, we denote by $\Delta'_H$ the Newton polygon of
the polynomial $\Hess(Q^{\Delta'}) $.

\begin{prop}\label{patch cp}
If $\Delta'$ is a cell of $\tau_\nu$ lying in the region $\{(x,y) | \
x\ge 2\ and \ y\ge 2\}$, then $ \Delta'_H$ is a cell of the
subdivision $\tau_{\widetilde \nu}$ and $\Hess(Q_t)^{
\Delta'_H}=\Hess(Q^{\Delta'}) $.
\end{prop}

\begin{proof} 
It is a standard fact (\cite{GKZ} p. 193) that the Newton polytope of
the product of two polynomials corresponds to the Minkowski sum of the
Newton polytopes of the factors. Let $\Gamma_1$ and $\Gamma_2$ be two
identical (up to translation) polytopes, and let $\Gamma_1 \oplus
\Gamma_2$ be their Minkowski sum.  Consider the natural map $\phi:
\Gamma_1 \times \Gamma_2 \rightarrow \Gamma_1 \oplus \Gamma_2$.  The
polytope $\Gamma_1 \oplus \Gamma_2$ is (up to translation) twice the
polytope $\Gamma_i$ and there are natural bijections $\iota_i$ between
the faces (of any dimension) of $\Gamma_1 \oplus \Gamma_2$ and the
faces of $\Gamma_i$. It is not difficult to see that for any face $F$
the preimage $\phi^{-1}(F)$ is exactly $\iota_1(F)\times \iota_2(F)$.

Given a vector $\vec u$ in $\RR^3$, we denote by $tr_{\vec u}$ the
translation along $\vec u$. We denote also respectively by $p_1, p_2 $
and $p_3$ the polynomials $\frac{\partial^2 Q_t}{\partial X^2}$,
$\frac{\partial^2 Q_t}{\partial Y^2}$ and $\frac{\partial^2
Q_t}{\partial X \partial Y}$.  Hence the polynomial $\Hess(Q_t)$ is
the difference of the two products $p_1 p_2$ and $p_3^2$.

A face of $\Delta_3(Q_t)$ lying in the region $\{(x,y,z) | \ x\ge 2\
\mbox{ and } \ y\ge 2\}$ is also a face of
$tr_{(2,0,0)}(\Delta_3(p_1))$, $tr_{(0,2,0)}(\Delta_3(p_2))$ and
$tr_{(1,1,0)}(\Delta_3(p_3))$.  Since these three polytopes are
contained in $\Delta_3(Q_t)$, the result follows immediately from the
above discussion about Minkowski sum of two identical polytopes and
the fact that coefficients of $p_1 p_2$ and $p_3^2$ corresponding to a
vertex $v$ of $\Delta_3(Q_t)$ are different as long as $v$ has a
nonzero first or second coordinate.
\end{proof}

\section{Construction}\label{proof}

Here we apply Proposition \ref{patch cp} to
glue Hessian curves. 
We first construct  pieces we need for patchworking.

\begin{lemma}\label{bound}
For any $i\ge 2$ and $j\ge 2$, the curves with equation $\Hess(X^iY^j(1+Y))=0$,
$\Hess(X^iY^j(X+Y))=0$, and $\Hess(X^iY^j(X+Y^2))=0$  do
not have any real points in $(\RR^*)^2$.
\end{lemma}
\begin{proof}
Up to division by powers of $x$ and $y$ these polynomial are of degree
2 with negative discriminant.
\end{proof}
By symmetry, the curve with equation $\Hess(X^iY^j(1+X))=0$ does not
have any real points in $(\RR^*)^2$.

\begin{lemma}\label{pieces}
For any $i\ge 2$ and $j\ge 2$, the real point set in $(\RR^*)^2$ of
the curves with equation $\Hess(X^iY^j(X+Y+Y^2))=0$,
$\Hess(X^iY^j(XY+X+Y^2))=0$, and $\Hess(X^iY^j(1+X+Y))=0$ consists of
one compact smooth oval.
\end{lemma}
\begin{proof}
According to Proposition \ref{patch cp} and Lemma \ref{bound}, these
three curves can only have compact connected component in $(\RR^*)^2$. 
Up to division by powers of $X$ and $Y$, the discriminant with respect
to the variable $X$ of these polynomials are degree 2 polynomials in
$Y$ with positive discriminant. They thus have exactly two distinct
real roots which attest the existence of exactly one oval for each
Hessian curve.
\end{proof}

To prove Theorem~\ref{theorem} we apply Viro's patchworking theorem to
a polynomial $Q_t(X,Y)$ whose truncation on the polygons of $\tau_\nu$
are the polynomials $X^iY^j(X+Y+Y^2)$, $X^iY^j(XY+X+Y^2)$ and
$X^kY^2(1+X+Y)$ for $2\le i,j\le d-2$ and $2\le k\le d-1$.
Proposition~\ref{patch cp} and Lemma~\ref{pieces} insure that, for
sufficiently small positive $t$, the Hessian curve of $Q_t(X,Y)$ has
at least $(d-2)^2$ smooth compact connected components.

\small
\def\rightmark

\end{document}